\documentclass[a4, 12pt]{amsart}
\usepackage{amssymb}
\usepackage{amstext}
\usepackage{amsmath}
\usepackage{amscd}
\usepackage{latexsym}
\usepackage{amsfonts}
\usepackage{color}
\usepackage{enumerate} 
\usepackage{textcomp}
\usepackage[all]{xy}
\usepackage[dvips]{graphicx}
\usepackage{colortbl}

\theoremstyle{plain}
\newtheorem{thm}{Theorem}[section]
\newtheorem*{thm*}{Theorem}
\newtheorem*{cor*}{Corollary}
\newtheorem*{defn*}{Definition}

\newtheorem{prop}[thm]{Proposition}
\newtheorem{lem}[thm]{Lemma}
\newtheorem{cor}[thm]{Corollary}
\newtheorem{claim}{Claim}
\newtheorem*{claim*}{Claim}

\theoremstyle{definition}
\newtheorem{defn}[thm]{Definition}
\newtheorem{ex}[thm]{Example}
\newtheorem{rem}[thm]{Remark}
\newtheorem{fact}[thm]{Fact}
\newtheorem{conj}[thm]{Conjecture}

\newtheorem{setting}[thm]{Setting}
\newtheorem{condition}[thm]{Condition}
\theoremstyle{remark}

\theoremstyle{remark}
\newtheorem*{ac}{Acknowledgments}
\numberwithin{equation}{thm}

\newcommand{\rmc}{\mathrm{c}}

\newcommand{\rme}{\mathrm{e}}

\newcommand{\rmr}{\mathrm{r}}

\newcommand{\rmv}{\mathrm{v}}

\newcommand{\rmK}{\mathrm{K}}

\newcommand{\rmQ}{\mathrm{Q}}

\newcommand{\rmT}{\mathrm{T}}

\newcommand{\calC}{\mathcal{C}}

\newcommand{\calF}{\mathcal{F}}

\newcommand{\calM}{\mathcal{M}}

\newcommand{\calS}{\mathcal{S}}

\newcommand{\fkc}{\mathfrak{c}}

\newcommand{\fkm}{\mathfrak{m}}
\newcommand{\fkn}{\mathfrak{n}}

\newcommand{\fkp}{\mathfrak{p}}

\def\Spec{\operatorname{Spec}}
\def\Hom{\operatorname{Hom}}
\def\Max{\operatorname{Max}}
\def\Ext{\operatorname{Ext}}
\def\Tor{\operatorname{Tor}}
\def\Ker{\operatorname{Ker}}
\def\Im{\operatorname{Im}}
\def\a{\mathfrak a}
\def\c{\mathfrak c}
\def\e{\operatorname{e}}
\def\m{\mathfrak m}
\def\n{\mathfrak n}

\begin{document}

\setlength{\baselineskip}{14.95pt}

\title{Huneke--Wiegand conjecture and change of rings}
\pagestyle{plain}
\author{Shiro Goto}
\address{Department of Mathematics, School of Science and Technology, Meiji University, 1-1-1 Higashi-mita, Tama-ku, Kawasaki 214-8571, Japan}
\email{goto@math.meiji.ac.jp}
\author{Ryo Takahashi}
\address{Graduate School of Mathematics, Nagoya University, Furocho, Chikusaku, Nagoya 464-8602, Japan}
\email{takahashi@math.nagoya-u.ac.jp}
\urladdr{http://www.math.nagoya-u.ac.jp/~takahashi/}
\author{Naoki Taniguchi}
\address{Department of Mathematics, School of Science and Technology, Meiji University, 1-1-1 Higashi-mita, Tama-ku, Kawasaki 214-8571, Japan}\email{taniguti@math.meiji.ac.jp}
\author{Hoang Le Truong}
\address{Institute of Mathematics, Vietnam Academy of Science and Technology, 18 Hoang Quoc Viet Road, 10307 Hanoi, Vietnam}
\email{hltruong@math.ac.vn}
\thanks{The first author was partially supported by JSPS Grant-in-Aid for Scientific Research (C) 25400051. The second author was partially supported by JSPS Grant-in-Aid for Scientific Research (C) 25400038}
\begin{abstract}
Let $R$ be a Cohen--Macaulay local ring of dimension one with a canonical module $\rmK_R$.
Let $I$ be a faithful ideal of $R$.
We explore the problem of when $I \otimes_RI^{\vee}$ is torsionfree, where $I^{\vee} = \Hom_R(I, \rmK_R)$.
We prove that if $R$ has multiplicity at most $6$, then $I$ is isomorphic to $R$ or $\rmK_R$ as an $R$-module, once $I\otimes_RI^{\vee}$ is torsionfree.
This result is applied to monomial ideals of numerical semigroup rings.
A higher dimensional assertion is also discussed.
\end{abstract}

\maketitle




\section{Introduction}

Let $M$ and $N$ be finitely generated modules over an integral domain $R$ and assume that both modules $M$ and $N$ are torsionfree. The destination of this research is to get an answer for the question of when the tensor product $M\otimes_RN$ is torsionfree. Our interest dates back to the following conjecture.

\begin{conj}[Huneke--Wiegand conjecture \cite{HW}]\label{HWC}
Let $R$ be a Gorenstein local domain. Let $M$ be a maximal Cohen--Macaulay $R$-module. If $M \otimes_R \Hom_R(M, R)$ is torsionfree, then $M$ is free.
\end{conj}

\noindent
Conjecture \ref{HWC} classically holds true, when the base ring $R$ is integrally closed (\cite[Proposition 3.3]{A}) and derived from the Auslander--Reiten conjecture (see \cite{CT}), which is one of the most important conjectures in ring theory. Conjecture \ref{HWC} is closely related to the Auslander--Reiten conjecture itself; in fact, Conjecture \ref{HWC} implies the Auslander--Reiten conjecture over Gorenstein local domains of positive dimension (\cite[Proposition 5.10]{CT}). C. Huneke and R. Wiegand \cite{HW} proved that it holds true if $R$ is a hypersurface, and showed also that Conjecture \ref{HWC} is reduced to the case where $\dim R = 1$. The problem is, however, still open in general, and no one has a complete answer to the following Conjecture \ref{1.2}, even in the case where $R$ is a complete intersection, or in the rather special case where $R$ is a numerical semigroup ring; see \cite{Constapel, H1, H2, HS}. The reader can consult \cite{GL, L} for the  recent major progress on numerical semigroup rings.

\begin{conj}{\label{1.2}}
Let $R$ be a Gorenstein local domain of dimension one and $I$ an ideal of $R$. If $I\otimes_R\Hom_R(I, R)$ is torsionfree, then $I$ is a principal ideal.
\end{conj}

In this paper we are interested in considering this conjecture in the Cohen--Macaulay case by replacing $\Hom_R(I,R)$ with $\Hom_R(I,\rmK_R)$, where $\rmK_R$ stands for the canonical module of $R$. Our working hypothesis is the following Conjecture \ref{1.3}. One of the advantages of such a  modification is the usage of the symmetry between $I$ and $\Hom_R(I,\rmK_R)$ and the other one is the possible  change of rings (see Proposition \ref{2.2}). Of course, when the ring $R$ is Gorenstein, Conjecture \ref{1.3} is the same as Conjecture \ref{1.2}, since $\rmK_R \cong R$.

\begin{conj}\label{1.3}
Let $R$ be a Cohen--Macaulay local ring of dimension one and assume that $R$ possesses a canonical module $\rmK_R$. Let $I$ be a faithful ideal of $R$. If $I \otimes_R \Hom_R(I, \rm{K}_R)$ is torsionfree, then $I$ is isomorphic to either $R$ or $\rmK_R$ as an $R$-module.
\end{conj}

\noindent
We should note here, in advance, that Conjecture \ref{1.3} is not true in general; later we shall give a counterexample. Nevertheless, the inquiry into the truth of Conjecture \ref{1.3} will make a certain amount of progress also in the study of Conjecture \ref{1.2}, which we would like to report in this paper.

Let us state our results, explaining how this paper is organized.
The following is the main result of our paper, which leads to Corollary \ref{1.5} of higher dimension.

\begin{thm}\label{1.4}
Let $R$ be a Cohen--Macaulay local ring of dimension one having a canonical module $\rmK_R$.
Let $I$ be a faithful ideal of $R$.
Set $r=\mu_R(I)$ and $s=\mu_R(\Hom_R(I,\rmK_R))$.
\begin{enumerate}[\rm(1)]
\item
Assume that the canonical map $I \otimes_R\Hom_R(I,\rmK_R)\to \rmK_R$ is an isomorphism.
If $r, s \ge 2$, then $\e(R) > (r+1)s \ge 6$.
\item
Suppose that $I \otimes_R\Hom_R(I,\rmK_R)$ is torsionfree.
If $\e(R) \le 6$, then $I$ is isomorphic to either $R$ or $\rmK_R$.
\end{enumerate}
\end{thm}

\noindent
Here, $\mu_R(*)$ denotes the number of elements in a minimal system of generators, and $\e(*)$ stands for the multiplicity with respect to the maximal ideal of $R$.


\begin{cor}\label{1.5}
Let $R$ be a Cohen--Macaulay local ring with $\dim R \ge 1$.
Assume that for every height one prime ideal $\fkp$ the local ring $R_\fkp$ is Gorenstein and $\e (R_\fkp) \le 6$. Let $I$ be a faithful ideal of $R$. If $I\otimes_R\Hom_R(I,R)$ is reflexive, then $I$ is a principal ideal.
\end{cor}

We shall prove Theorem \ref{1.4} and Corollary \ref{1.5} in Section 3. Section 2 is devoted to some preliminaries, which we need to prove Theorem \ref{1.4} and Corollary \ref{1.5}.

In Section 4 (and partly in Section 3) we study numerical semigroup rings.
Let $V = k[[t]]$ be a formal power series ring over a field $k$.
Let $0 < a_1 < a_2< \cdots <a_{\ell}$ be integers such that $\mathrm{gcd}(a_1, a_2, \ldots, a_{\ell}) = 1$, and let
$$
R = k[[t^{a_1}, t^{a_2}, \ldots, t^{a_\ell}]]
$$
be the subring of $V$ generated by $t^{a_1}, t^{a_2}, \ldots, t^{a_\ell}$, which is called a numerical semigroup ring over $k$.
With this notation we have the following, which we shall prove in Section 3 (Corollary \ref{3.5}).

\begin{thm}\label{1.6}
Let $R=k[[t^a, t^{a+1}, \ldots, t^{2a-2}]]$ $(a\ge3)$ and $I$ an ideal of $R$. If the $R$-module $I\otimes_R\Hom_R(I,R)$ is torsionfree, then $I$ is a principal ideal.
\end{thm}

\noindent
We notice that Theorem \ref{1.6} gives a new class of Gorenstein local domains for which Conjecture \ref{1.2} holds true. In fact, the ring $R$ is a Gorenstein local ring which is not a complete intersection, if $a \ge 5$ (see Example \ref{3.6}).

In Sections 4 and 5 we study monomial ideals, that is, ideals generated by monomials in $t$. The main result is the following, which covers \cite[Main Theorem]{H1} in the case where $R$ is a Gorenstein ring.

\begin{thm}\label{1.8}
Let $R=k[[t^{a_1}, t^{a_2}, \ldots, t^{a_\ell}]]$ be a numerical semigroup ring over a field $k$ and assume that $\e(R) \le 7$. Let $I\ne(0)$ be a monomial ideal of $R$. If $I\otimes_R\Hom_R(I,\rmK_R)$ is torsionfree, then one has either $I\cong R$ or $I\cong \rm{K}_R$.
\end{thm}

Unfortunately, Theorem \ref{1.8} and hence Conjecture \ref{1.3} are  no longer true, when $\e (R) =9$ (see Example \ref{7.1}). We are still not sure whether the assertion stated in Theorem \ref{1.8} is true in general, when  $\e (R) = 8$. We actually have monomial ideals $I$ in several numerical semigroup rings $R$ with $\e (R) = 8$, which satisfy the equalities
$$\mu_R(I){\cdot}\mu_R(\Hom_R(I,\rmK_R))=\mu_R(\rmK_R)\ \ \text{and}\ \ I{\cdot}(\rmK_R : I) = \rmK_R.$$ However, as far as we know, the $R$-modules $I \otimes_R\Hom_R(I,\rmK_R)$ do have non-trivial torsions for those ideals $I$.

In Section 6 we note an elementary method to compute the torsion part $\rmT(I\otimes_RJ)$ of $I\otimes_RJ$ for some ideals $I, J$ of $R$, and apply it in Section 7 to explore concrete examples.

In what follows, unless otherwise specified, let $R$ be a Cohen--Macaulay local ring  with  maximal ideal $\m$. We set $F=\rmQ (R)$, the total ring of fractions of $R$. For each finitely generated $R$-module $M$, let  $\mu_R(M)$ and $\ell_R(M)$ denote,  respectively,  the number of elements in a minimal system of generators of $M$ and the length of $M$. For each Cohen--Macaulay $R$-module $M$, we denote by  $\rmr_R(M)$ the Cohen--Macaulay type of $M$ (see \cite[Definition 1.20]{HK2}).


\section{Change of rings}

The purpose of this section is to summarize some preliminaries, which we need throughout this paper.

Let $R$ be a Cohen--Macaulay local ring  with  maximal ideal $\m$ and $\dim R = 1$. Let $F=\rmQ (R)$ stand for the total ring of fractions of $R$ and let $\calF$ denote the set of fractional ideals $I$ of $R$ such that $FI = F$. Assume that $R$ possesses a canonical module $\rmK_R$. This condition is equivalent to saying that $R$ is a homomorphic image of Gorenstein local ring (\cite{R}). For each $R$-module $M$ we set $M^{\vee} = \Hom_R(M,\rmK_R)$.

Let $I \in \calF$.
Denote by
$$
t : I \otimes_RI^{\vee} \to \rmK_R
$$
the $R$-linear map given by $t(x \otimes f) = f(x)$ for $x \in I$ and $f \in I^{\vee}$.
Let $\alpha : I\otimes_RI^{\vee} \to F \otimes_R(I \otimes_RI^{\vee})$ and $\iota : \rmK_R \to F\otimes_R\rmK_R$ be the maps given by $\alpha(y)=1\otimes y$ and $\iota(z)=1\otimes z$ for $y\in I\otimes_RI^{\vee}$ and $z\in\rmK_R$.
Let
\begin{multline*}
\phi:F\otimes_R(I\otimes_R I^{\vee})
\xrightarrow{\sim}(F\otimes_RI)\otimes_F\Hom_F(F \otimes_RI, F\otimes_R\rmK_R)\\
\xrightarrow{\sim}F\otimes_F \Hom_F(F,F\otimes_R\rmK_R)
\xrightarrow{\sim}F\otimes_R\rmK_R
\end{multline*}
be the composition of natural isomorphisms.
Then the diagram
$$
\begin{CD}
F\otimes_R(I\otimes_R I^{\vee}) @>{\phi}>> F\otimes_R\rmK_R \\
@A{\alpha}AA @A{\iota}AA \\
I \otimes_R I^{\vee} @>{t}>> \rmK_R
\end{CD}
$$
is commutative.
As $\iota$ is injective, we have $\Ker\alpha = \Ker t$.
Hence the torsion part $\rmT(I\otimes_RI^{\vee})$ of the $R$-module $I\otimes_RI^{\vee}$ is given by $$\rmT(I\otimes_RI^{\vee}) = \Ker t$$ and we get the following.

\begin{lem}\label{2.1} The $R$-module $I \otimes_R I^{\vee}$ is torsionfree if and only if the map $t : I \otimes_R I^{\vee} \longrightarrow {\rm K}_R$ is injective.
\end{lem}

We set $L = \Im(I \otimes_R I^{\vee} \xrightarrow{t} {\rm K}_R)$ and $T =\rmT(I \otimes_R I^{\vee})$. Then $T^{\vee} = (0)$ since $\ell_R(T) < \infty$.
Taking the $\rmK_R$-dual of the short exact sequence $0 \to T \to I\otimes_R I^{\vee} \xrightarrow{t} L \to 0$, we have $L^{\vee} = (I \otimes_R I^{\vee})^\vee$. Hence the equalities
$$L^{\vee} = (I \otimes_R I^{\vee})^\vee = \Hom_R(I, I^{\vee \vee})= I:I
$$
follow. Recall that $B = I:I$ forms a subring of $F$ which is a module-finite over $R$. 

We now take an arbitrary intermediate ring $R \subseteq S \subseteq B$. Then $I$ is also a fractional ideal of $S$. We set $\rmK_S = S^\vee$ and remember that $L=L^{\vee \vee}$ (\cite[Satz 6.1]{HK2}). Then since $L^{\vee} = B$, we have  
\begin{align*}
&L=L^{\vee \vee} = B^\vee = \rmK_B \subseteq S^\vee=\rmK_S\quad\text{and}\\
&\Hom_S(I,\rmK_S) = \Hom_S(I, \Hom_R(S,\rmK_R)) \cong \Hom_R(I\otimes_SS,\rmK_R) = \Hom_R(I,\rmK_R).
\end{align*}
Let us identify $I^\vee = \Hom_S(I,\rmK_S)$  and we consider  the commutative diagram
\[
\xymatrix{
0 &&&\\
I \otimes_S \Hom_S(I, \rmK_S) \ar[u]  &\ar[r]^{t_S} & &\rmK_S  \\
I \otimes_R I^{\vee} \ar[u]^{\rho} &\ar[r]^{t} & &L \ar[u]^{\iota} & \ar[r]& &0
}
\]
where $\iota : L \to \rmK_S$ is the embedding and $\rho:I \otimes_R I^{\vee} \to I\otimes_S \Hom_S(I, \rmK_S)$ denotes the $R$-linear map defined by $\rho (x\otimes f)= x\otimes f$ for $x \in I$ and $f \in I^{\vee}$. Suppose now that $I\otimes_RI^{\vee}$ is torsionfree. Then since the map $t:I\otimes_RI^{\vee} \to L$ is bijective by Lemma \ref{2.1}, the map $\rho : I\otimes_RI^{\vee} \to I\otimes_S\Hom_S(I,\rmK_S)$ is bijective, whence the $S$-module $I\otimes_S\Hom_S(I,\rmK_S)$ is also torsionfree, because the map $t_S : I \otimes_S\Hom_S(I,\rmK_S) \to \rmK_S$ is injective. Thus we get the following, where  the last assertion comes from the fact  that $L = \rmK_B$.

\begin{lem}\label{2.1a} Let $I \in \calF$ and suppose that $I \otimes_R I^{\vee}$ is torsionfree. Let $R \subseteq S \subseteq B$ be an intermediate ring between $R$ and $B$, where $B = I:I$. Then $I \otimes_S \Hom_S(I, \rmK_S)$ is a torsionfree $S$-module and the canonical map $\rho : I\otimes_RI^{\vee} \to I\otimes_S\Hom_S(I,\rmK_S)$ is bijective. In particular, if we take $S = B$, then the map $$t_B : I \otimes_B \Hom_B(I, \rmK_B) \to \rmK_B, ~~x \otimes f \mapsto f(x)$$ is an isomorphism of $B$-modules.
\end{lem}

The following is the key in our arguments.

\begin{prop}[{\bf Principle of change of rings}]\label{2.2}
Let   $I \in \calF$  and assume that  $I\otimes_R I^{\vee}$  is torsionfree. If there exists an intermediate ring $R\subseteq S \subseteq B$ such that $I \cong S$ or $I \cong \rmK_S$ as an $S$-module, then $I \cong R$ or $I \cong \rmK_R$ as an  $R$-module.
\end{prop}

\begin{proof}
Suppose that $I \cong S$ as an $S$-module and consider the isomorphisms
$$
I\otimes_R I^\vee \overset{\rho}{\cong} I \otimes_S\Hom_S(I,\rmK_S)\cong \Hom_S(I, \rmK_S)\cong I^\vee
$$
of $R$-modules. We then have $
\mu_R(I){\cdot}\mu_R(I^\vee) = \mu_R(I^\vee),
$ so that $I \cong R$ as an $R$-module, since $\mu_R(I)=1$. Suppose that $I \cong \rmK_S$ as an $S$-module. Then because $S\cong \Hom_S(\rmK_S, \rmK_S)$ (\cite[Satz 6.1 d) 3)]{HK2}), we get the isomorphisms
$$
I\otimes_R I^\vee \overset{\rho}{\cong} I \otimes_S \Hom_S(I, \rmK_S) \cong I
$$
of $R$-modules. Hence  $\mu_R(I^\vee)=1$, so that $I \cong \rmK_R$ as an $R$-module, because $I \cong I^{\vee \vee}$. 
\end{proof}

We close this section with the following.

\begin{prop}
Let $I$ be an $\fkm$-primary ideal of $R$ and assume that $I^2 = aI$ for some $a \in I$. If $I\otimes_RI^{\vee}$ is torsionfree, then $I = aR$.
\end{prop}

\begin{proof}
We have $a^{-1}I \subseteq I : I = B$, since $I^2 = aI$.  Therefore $I \cong B$ as a $B$-module, because $I = aB$, so that $I\otimes_RI^{\vee} \overset{\rho}{\cong} I \otimes_B\Hom_B(I,\rmK_B) \cong \Hom_B(I, \rmK_B) \cong I^{\vee}$. Thus $I \cong R$ as an $R$-module. Hence $I = aR$, because $B = R$.
\end{proof}


\section{Proof of Theorem \ref{1.4}}
The purpose of this section is to prove Theorem \ref{1.4}. We maintain the same notation and terminology as in Section 2.

\begin{proof}[Proof of assertion (1) of Theorem \ref{1.4}]
Enlarging the residue class field $R/\fkm$ of $R$, without loss of generality we may assume that the field $R/\m$ is infinite. Choose $f \in \m$ so that $fR$ is a reduction of $\m$. We set $S = R/fR$, $\n= \m/fR$, and $M = I/fI$. Hence $\mu_S(M)=r$ and $\rmr_S(M)=\ell_S((0):_M \n)=s$ by \cite[Bemerkung 1.21 a), Satz 6.10]{HK2} (here $\rmr_S(M)$ denotes the Cohen--Macaulay type of $M$). We  write $M = Sx_1+Sx_2+\cdots + Sx_r$ with $x_i \in M$ and consider the following presentation
$$(\sharp_0)\ \ \ \ 
0 \to X \to S^{\oplus r} \overset{\varphi}{\longrightarrow} M \to 0
$$
 of the $S$-module $M$, where $\varphi$ denotes the $S$-linear map defined by $\varphi(\mathbf{e}_j)=x_j$ for $1 \le \forall j \le r$ (here $\{\mathbf{e}_j\}_{1 \le j \le r}$ is the standard basis of $S^{\oplus r}$). Then, taking the $K_S$-dual (denoted by $[*]^\vee$ again) and the $M$-dual respectively of the above presentation $(\sharp_0)$, we get the following two exact sequences
$$(\sharp_1)\ \ \ \ 
0 \to M^{\vee} \to \rmK_S^{\oplus r} \to X^{\vee} \to 0,
$$
$$(\sharp_2)\ \ \ \ 
0 \to \Hom_S(M, M) \to M^{\oplus r}\to \Hom_S(X, M)
$$
of $S$-modules. Remember that $I\otimes_RI^\vee \overset{t}{\cong} \rmK_R$ and we have $$M\otimes_SM^\vee \cong S\otimes_R (I\otimes_RI^\vee) \overset{S\otimes_Rt}{\cong} S\otimes_R\rmK_R = \rmK_S,$$ because $S\otimes_RI^\vee = M^\vee$ and $S\otimes_R\rmK_R = \rmK_S$ (\cite[Lemma 6.5, Korollar 6.3]{HK2}). Hence $$S = \Hom_S(\rmK_S,\rmK_S) \cong \Hom_S(M\otimes_SM^\vee, \rmK_S) = \Hom_S(M, M^{\vee\vee}) = \Hom_S(M,M),$$ so that exact sequence $(\sharp_2)$ gives rise to the exact sequence 
$$(\sharp_3)\ \ \ \  
0 \to S \overset{\psi}{\longrightarrow} M^{\oplus r}\to \Hom_S(X, M),
$$
where $\psi = {}^t \varphi$ is the transpose of $\varphi$, satisfying
$\psi(1) = (x_1, x_2, \ldots, x_r) \in M^{\oplus r}$.

We set $q = \mu_S(X^{\vee})~(= \ell_S((0):_X\n))$ and $e = \rme (R)$.
Then by $(\sharp_0)$ we get
$$\ell_S(X) = r{\cdot}\ell_S(S) - \ell_S(M) = re-e=(r-1)e,$$
since $\ell_S(S) = \e (R)$ and $\ell_S(M) =  \e_{fR}^0(I) = \e_{fR}^0(R)=\e(R)$, where $\e_{fR}^0(I)$ and $\e_{fR}^0(R)$ denote respectively the multiplicity of $I$ and $R$ with respect to $fR$. On the other hand, by exact sequence $(\sharp_1)$ we have 
$$q = \mu_S(X^\vee) \ge \mu_S(\rmK_S^{\oplus r})- \mu_S(M^\vee) = r{\cdot}\mu_S(\rmK_S) - \rmr_S(M).
$$
Because $I \otimes_RI^\vee \cong \rmK_R$ and $
\mu_S(\rmK_S) = \rmr(S) = \rmr(R) = \mu_R(\rmK_R)
$ (\cite[Korollar 6.11]{HK2}), 
we get $\mu_S(\rmK_S) = rs$, whence
$$
(r-1)e = \ell_S(X) \ge \ell_S((0):_X \n) = q \ge r^2s-s = s(r^2-1).
$$
Thus $e \ge s(r+1)$, since $r, s \ge 2$.

Suppose now that $e = s(r+1)$. Then since $\ell_S(X)= \ell_S((0):_X \n)$, we get $\n{\cdot}X = (0)$, so that $\n{\cdot}\Hom_S(X,M) = (0)$. Therefore $\n {\cdot}M^{\oplus r}\subseteq S{\cdot}(x_1, x_2, \ldots, x_r)$ by exact sequence $(\sharp_3)$. Let $1 \le i \le r$, $f \in M$, and $z \in \fkn$ and write
$$z{\cdot}(0,\ldots, 0,\overset{\overset{i}{\vee}}{f}, \ldots, 0) = v{\cdot}(x_1, x_2, \ldots, x_r)$$
with $v \in S$. Then since  
$zf = vx_i$ and $0 = vx_j$ if $j \ne i$, we get
$\n M \subseteq \a_iM$, where $\a_i = (0):(x_j \mid 1 \le j \le r, ~j \ne i)$. Notice that $\a_i \ne S$, since $r = \mu_S(M) \ge 2$. Therefore
 $\n M = \a_i M$ for all $1 \le i \le r$, so that $\n^2M=(\a_1\a_2)M$, whence $\n^2M = (0)$ because $\a_1\a_2 \subseteq (0):(x_i \mid 1 \le i \le r) = (0)$ (remember that $M$ is a faithful $S$-module; see exact sequence $(\sharp_3)$). Thus $\n M \subseteq (0):_M \n$. Consequently
$$
s = \rmr_S(M) = \ell_S((0):_M\n) \ge \ell_S(\n M) = \ell_S(M)- \ell_S(M/\n M) = e -r = s(r+1)-r.
$$
Hence $0 \ge rs-r = r(s-1)$, which is impossible because $r, s \ge 2$.
The proof of assertion (1) of Theorem \ref{1.4} is now completed.
\end{proof}

Let us prove assertion (2) of Theorem \ref{1.4}.

\begin{proof}[Proof of assertion (2) of Theorem \ref{1.4}]
Enlarging the residue class field $R/\fkm$ of $R$ if necessary and passing to the $\fkm$-adic completion of $R$, without loss of generality  we may assume that $R$ is complete, possessing  infinite residue class field. Let $B = I:I$. Then since $B$ is a module-finite extension of $R$, we get the canonical decomposition $$B \cong \prod_{\fkn \in \Max B}B_\fkn$$
of the ring $B$, 
where $\Max B$ denotes the set of maximal ideals of $B$. Remember that by Lemma \ref{2.1a} the homomorphism of $B$-modules $$t_B : I\otimes_B\Hom_B(I,\rmK_B) \to \rmK_B, \ \ x \otimes f \mapsto f(x)$$ is an isomorphism. Hence for each $\fkn \in \Max B$ we get the canonical isomorphism 
$$
(\sharp)\ \ \  \ \ \ I_\fkn \otimes_{B_\fkn}\Hom_{B_{\fkn}}(I_{\fkn},\rmK_{B_\fkn}) \overset{t_{B_\n}}{\cong}  \rmK_{B_\fkn}
$$
of $B_\fkn$-modules, since $[\rmK_B]_\fkn = \rmK_{B_\fkn}$ (\cite[Satz 5.22]{HK2}). We now choose $f \in \fkm$ so that $fR$ is a reduction of $\fkm$. Hence
$$
6 \ge \e(R) =\e_{fR}^0(R) = \e_{fR}^0(B)= \ell_R(B/fB).
$$
Therefore, because 
$$
\textstyle
\ell_R(B/fB) = \sum_{\fkn \in \Max B}\ell_R(B_\fkn/fB_\fkn) 
\ge \sum_{\fkn \in \Max B}\ell_{B_\fkn}(B_\fkn/fB_\fkn)
\ge \sum_{\fkn \in \Max B}\e(B_\fkn),
$$
we have $\e(B_\fkn) \le 6$ for each $\fkn \in \Max B$. Thus by assertion (1) of Theorem \ref{1.4},
$$
I_{\fkn} \cong B_\fkn \ \ \text{or} \ \ I_{\fkn} \cong \rmK_{B_\fkn}
$$
as an $B_\n$-module. Hence, thanks to  Lemma \ref{2.2}, assertion (2) of Theorem \ref{1.4} now  follows from the following claim. We actually have  $I \cong B$ as a $B$-module if case (i) occurs and  $I \cong \rmK_B$ as a $B$-module  if case (ii) occurs.

\begin{claim} One of the following two cases must occur.
\begin{enumerate}
\item[$(\mathrm{i})$] $I_\fkn \cong B_\fkn$ for every $\fkn \in \Max B$.
\item[$(\mathrm{ii})$] $I_\fkn \cong \rmK_{B_\fkn} ~(= [\rmK_B]_\n)$ for every $\fkn \in \Max B$.
\end{enumerate}
\end{claim}

\noindent
{\bf {\it Proof of Claim 1.}}
Assume the contrary. Then $B$ is neither a local  ring nor a Gorenstein ring.  We firstly choose $\fkn_1 \in \Max B$ so that $B_{\fkn_1}$ is not a Gorenstein ring. Then $\e (B_{\fkn_1}) \ge 3$, because $B_{\fkn_1}$ is a hypersurface if  $\e(B_{\fkn_1}) \le 2$. Choose  $\fkn_2 \in \Max B \setminus \{\n_1\}$. Then since
$$\textstyle
6 \ge \sum_{\fkn \in \Max B}\e (B_{\fkn}) \ge \e(B_{\fkn_1}) +  \e(B_{\fkn_2}) \ge 4,
$$
we see $\sharp \Max B \le 4$. If $\sharp \Max B = 3$ or $4$, then $B_\n$ is a Gorenstein ring for each $\n \in \Max B\setminus \{ \n_1\}$ (since $\e(B_\n) \le 2$). Nevertheless, if $B_\n$ is a Gorenstein ring for every $\n \in \Max B\setminus \{ \n_1\}$, we then have $I_{\fkn} \cong B_{\fkn}\cong \rmK_{B_{\fkn}}$, so that  if $I_{\fkn_1} \cong B_{\fkn_1}$ then $I_{\fkn} \cong {B_{\fkn}}$ for every $\fkn \in \Max B$, and if $I_{{\fkn}_1}\cong \rmK_{B_{\fkn_1}}$ then $I_{\fkn} \cong \rmK_{B_\fkn}$ for every $\fkn \in \Max B$, which is impossible. Thus $\Max B = \{ \n_1, \n_2  \}$ and $B_{\n_2}$ is also not a Gorenstein ring. Without loss of generality we may assume that $I_{B_{\n_1}} \cong B_{\n_1}$ and $I_{B_{\n_2}} \cong \rmK_{B_{\n_2}}$. Hence 
$$
I \oplus I^{\vee} \cong B \oplus \rmK_B
$$
as a $B$-module, 
because $(I^{\vee})_{\n_1} \cong \rmK_{B_{\n_1}}$ and  $(I^{\vee})_{\n_2} \cong B_{\n_2}$. We set $L = \Im(I \otimes_R I^{\vee} \overset{t}{\longrightarrow} {\rm K}_R)$. Then $\mu_R(L) = rs$, as $L \cong I\otimes_RI^\vee$, while we get 
$$
I \oplus I^{\vee} \cong L^{\vee} \oplus L
$$
as an $R$-module, because  $L^{\vee} = I:I = B$ and hence $\rmK_B = L$ (see Section 2). Consequently, setting $q = \mu_R(B) \ge 1$, we have $r+s = q+rs$. Thus $1-q =(r-1)(s-1) \ge 0$, whence $q = 1$, so that $B= R$. This is impossible, because $B$ is not a local ring.
\end{proof}

The following is a direct consequence of Theorem \ref{1.4}  (2).

\begin{cor}\label{3.4}
Let $R$ be a Gorenstein local ring with $\dim R = 1$ and $\e(R) \le 6$. Let $I$ be a faithful ideal of $R$. If $I \otimes_R\Hom_R(I,R)$ is torsionfree, then $I$ is a principal ideal.
\end{cor}

Let us  give a proof of Corollary \ref{1.5}.

\begin{proof}[Proof of Corollary \ref{1.5}]
Assume the contrary and choose $R$ so that $d = \dim R$ is the smallest among counterexamples. Then $d \ge 2$ by  Theorem \ref{1.4} (2). Let $\fkp \in \Spec R \setminus \{\fkm\}$. Then because $I_\fkp$ is a faithful ideal of $R_\fkp$ and $I_\fkp \otimes_{R_\fkp}\Hom_{R_\fkp}(I_{\fkp},R_\fkp) = [I\otimes_R\Hom_R(I,R)]_{\fkp}$ is $R_\fkp$-reflexive, the minimality of $d=\dim R$ shows $I_\fkp \cong R_\fkp$ as an $R_\fkp$-module. Hence by Auslander's theorem (see \cite[Theorem 3.4]{Celikbas-Dao}) $I$ is $R$-free, which is impossible. 
\end{proof}

Let $R$ be a Cohen--Macaulay local ring with maximal ideal $\fkm$ and $\dim R = 1$. We denote by  $\overline{R}$ the integral closure of $R$ in the total ring $\rmQ (R)$ of fractions. Assume that $\overline{R}$ is a finitely generated $R$-module. Then the $\fkm$-adic completion $\widehat{R}$ of $R$ is a reduced ring, so that $R$ possesses a canonical ideal $K$ (\cite[Satz 6.21]{HK2}), that is a fractional ideal of $R$ such that $\widehat{R}\otimes_RK \cong \rmK_{\widehat{R}}$ as an $\widehat{R}$-module. In particular, $R$ possesses a canonical module $K=\rmK_R$.  We furthermore have the following.

\begin{thm}\label{3.3} Let $R$ be as above and  let $I$ be a faithful fractional ideal of $R$. If $I \otimes_RI^{\vee}$ is torsionfree, then $I\cong R$ or $I\cong \rm{K}_R$ as an $R$-module.
\end{thm}

\begin{proof}
We may assume the residue class field of $R$ is infinite. We may assume $I \subseteq R$. Choose $f \in \fkm$ and $g \in I$ so that $\fkm = \fkm \overline{R} = f \overline{R}$ and $I \overline{R} = g \overline{R}$ (these choices are possible, since $\overline{R}$ is a principal ideal ring and the residue class field of $R$ is infinite). Then, because $g$ is invertible in $\rmQ (R)$ and $fR \subseteq \frac{f}{g}I \subseteq f\overline{R} = \fkm \overline{R},$
replacing $I$ with $\frac{f}{g}I$, without loss of generality we may assume $fR \subseteq I \subseteq \fkm \overline{R}=\fkm.$ We set $S = R/I$, $\fkn = \fkm/I$, $r = \mu_R(I)$, and $s = \mu_R(I^{\vee})$. Then $\fkn^2= (0)$, since $\fkm^2 = f \fkm \subseteq I$. Therefore $$\ell_S((0):_S\fkn) \ge \ell_S(\n) = \mu_S(\fkn) \ge \mu_R(\fkm) - r.$$ On the other hand, taking the $\rmK_R$-dual of the short exact sequence 
$0 \to I \to R \to S \to 0,$
 we get an epimorphism 
$I^{\vee} \to \Ext_R^1(S,\rmK_R).$
Therefore $s \ge \mu_R(\Ext_R^1(S,\rmK_R)) = \ell_S((0):_S\fkn)$ (\cite[Satz 6.10]{HK2}). Let $J = \rmK_R : I ~(\cong I^\vee)$ and set $L = IJ \subseteq \rmK_R$. Then $L \cong I\otimes_RI^\vee$ by Lemma \ref{2.1}. Hence $\e (R) \ge \mu_R(L)=rs$
(\cite[Chapter 3, 1.1. Theorem ]{S2}), because $\rmK_R$ is (and hence $L$ is) a fractional ideal of $R$. Thus
$$s \ge \ell_S((0):_S\fkn) \ge \mu_R(\fkm) - r = \e (R) -r \ge rs - r$$ 
(remember that $\mu_R(\fkm) = \e (R)$, since $\fkm^2 = f \fkm$; see \cite[Theorem 1]{S1}),
so that $1 \ge (r-1)(s-1).$
Consequently, if $r, s \ge 2$, then $r = s= 2$, whence $2 \ge \e (R) - 2$,  that is $\e (R) \le 4$. This violates Theorem \ref{1.4} (2). Thus $r = 1$ or $s = 1$, whence $I \cong R$ or $I \cong \rmK_R$ as an $R$-module.
\end{proof}

The following ring $R$ is  a Cohen--Macaulay local ring with $\dim R = 1$ and $\fkm \overline{R} \subseteq R$.

\begin{cor}
Let $S$ be a regular local ring with maximal ideal $\n$ and $n = \dim S > 0$. Let $x_1, x_2, \ldots, x_n$ be a regular system of parameters of $S$. For each $1 \le i \le n$ let $$\fkp_i = (x_j \mid 1 \le j \le n, ~j \ne i).$$ We set $R = S /\bigcap_{i=1}^n\fkp_i.$ Let $I$ be a faithful ideal of $R$. If $I\otimes_R\Hom_R(I,\rmK_R)$ is torsionfree, then $I \cong R$ or $I \cong \rmK_R$ as an $R$-module.
\end{cor}

The following result is proved similarly as Theorem \ref{3.3}. Let us note a brief proof, because the result might have its own significance. Let $\rmv (R) = \mu_R(\fkm)$ denote the embedding dimension of $R$.

\begin{thm}\label{3.3a} Let $R$ be a Cohen--Macaulay local ring with maximal ideal $\fkm$ and $\dim R = 1$. Assume that $R$ possesses a canonical module $\rmK_R$ and $\rmv (R) = \e (R)$. Let $I$ be a faithful ideal of $R$. We set  $r=\mu_R(I)$ and $s=\mu_R(\Hom_R(I,\rmK_R))$. If $rs = \rmr(R)$, then $I \cong R$ or $I \cong \rmK_R$ as an $R$-module.
\end{thm}

\begin{proof}
We may assume that $R/\fkm$ is infinite and $e = \e (R) > 1$. Choose $f \in \fkm$ so that $fR$ is a reduction of $\fkm$. Then $\fkm^2 = f\fkm$ and
$$rs= \rmr (R) = \ell_R((0):_{R/fR}\fkm)= \mu_R(\fkm/fR) = \rmv (R) - 1 = e - 1,$$
because $f \not\in \fkm^2$ and $e > 1$. We set $S = R/fR$, $\n = \fkm/fR$, and $M = I/fI$. Then since $\n^2 = (0)$, we get
$$s = \ell_S((0):_M\n) \ge \ell_S(\n M) = \ell_S(M) - \ell_S(M/\n M) = e - r.$$
Hence $0 \ge (r-1)(s-1)$, because $e = rs + 1$. Thus $I \cong R$ or $I \cong \rmK_R$ as an $R$-module.
\end{proof}

Let us examine numerical semigroup rings.

\begin{prop}\label{3.4a}
Let $R=k[[t^a, t^{a+1}, \ldots, t^{2a-1}]]~(a \ge 1)$ be the semigroup ring of the numerical semigroup $H = \left<a, a+1, \ldots, 2a-1 \right>$ over a field $k$. Let $I \ne (0)$ be an arbitrary ideal  of $R$. If $I\otimes_RI^{\vee}$ is torsionfree, then $I\cong R$ or $I \cong \rmK_R$ as an $R$-module.
\end{prop}

\begin{proof}
This is clear and follows from Theorem \ref{3.3}, since $\overline{R} = k[[t]]$ and $\fkm k[[t]] = \fkm$.
\end{proof}

\begin{cor}\label{3.5}
Let $R=k[[t^a, t^{a+1}, \ldots, t^{2a-2}]]$ $(a\ge3)$ be the semigroup ring of the numerical  semigroup $H=\left<a, a+1, \ldots, 2a - 2\right>$ over a field $k$ and let $I$ be an ideal of $R$. If $I\otimes_R\Hom_R(I,R)$ is torsionfree, then $I$ is principal.
\end{cor}

\begin{proof}
Notice that $R$ is a Gorenstein local ring with $R:\m=R+kt^{2a-1}$ (see \cite[Satz 3.3, Korollar 3.4]{HK2}). Suppose that $\mu_R(I)>1$ and set $B= I:I$. Then $R\subsetneq B$. In fact, $I\otimes_B \Hom_B(I, \rmK_B)\cong \rmK_B$ by Lemma \ref{2.1a}. Hence, if $B = R$, then $I$ is invertible, so that $I$ must be a principal ideal. Thus $R \subsetneq B$ and therefore  $t^{2a-1} \in B$, whence
$$
R \subseteq S=k[[t^a, t^{a+1}, \ldots, t^{2a-1}]] \subseteq B.
$$
Then  by Lemma \ref{2.1a}
$I\otimes_S \Hom_S(I, \rmK_S)$ is $S$-torsionfree, so that  by Proposition  \ref{3.4a} $I \cong S $ or $I \cong \rmK_S$ as an $S$-module. Hence $I\cong R$ by Proposition \ref{2.2}, which is impossible.
\end{proof}

\begin{rem}\label{3.6} Corollary \ref{3.5} gives a new class of one-dimensional Gorenstein local domains for which Conjecture \ref{1.2} holds true. For example, in Corollary \ref{3.5} take $a = 5$. Then $R=k[[t^5, t^6, t^7, t^8]]$ is not a complete intersection. In fact, let $P= k[[X,Y,Z, W]]$ be the formal power series ring over $k$ and let $\varphi: U \to k[[t]]$ be the $k$-algebra map defined by $\varphi (X) = t^5, \varphi(Y) = t^6, \varphi (Z) = t^7$, and $\varphi (W) = t^8$. We then have $
\Ker\varphi = (Y^2 - XZ, Z^2-YW, W^2-X^2Y, X^3-ZW, XW-YZ)
$ and $\mu_P(\Ker\varphi)=5$.
\end{rem}


\section{Numerical semigroup rings and monomial ideals}

We focus our attention on numerical semigroup rings. Let us fix some notation and terminology.

\begin{setting}\label{setting}
Let $0 < a_1 < a_2< \cdots <a_{\ell}$ be integers such that $\mathrm{gcd}(a_1, a_2, \ldots, a_{\ell}) = 1$. We set $H = \left< a_1, a_2, \ldots, a_{\ell} \right> =\{\sum_{i = 1}^\ell c_i a_i \mid 0 \le c_i \in \Bbb Z\}$  and 
$$R = k[[t^{a_1}, t^{a_2}, \ldots, t^{a_\ell}]]~~~\subseteq ~~~k[[t]],$$
where $V=k[[t]]$ is the formal power series ring over a field $k$. Let $\fkm = (t^{a_1}, t^{a_2}, \ldots, t^{a_\ell})$ be the maximal ideal of $R$. We set $\fkc = R : V$ and $c = \rmc(H)$, the conductor of $H$, whence $\fkc = t^cV$. Let $a = c -1$. We denote by $\calF$ the set of non-zero fractional ideals of $R$. 
\end{setting}

Notice that $R$ is a Cohen--Macaulay local ring  with $\dim R = 1$ and $V$ the normalization. We have $\e (R) = a_1=\mu_R(V)$.

\begin{defn}Let $I \in \calF$. Then $I$ is said to be a monomial ideal, if $I = \sum_{n \in \Lambda}Rt^n$ for some $\Lambda \subseteq \Bbb Z$.
\end{defn}

We denote by $\calM$ the set of monomial ideals $I \in \calF$. Remember that each $I \in \calM$ has a unique  finite subset $\Lambda$ of $\Bbb Z$ such that $\{t^n\}_{n \in \Lambda}$ forms a minimal system of generators for the $R$-module $I$. We are now going to explore Conjecture \ref{1.3} on $I \in \calM$. For the purpose, passing to the monomial ideal $t^{-q}I$ with $q = \min \Lambda$, we may assume
$R \subseteq I \subseteq V$.

A canonical ideal $\rmK_R$ of  $R$ is given by $$\rmK_R = \sum_{n \in \Bbb Z \setminus H}Rt^{a-n}$$ (\cite[Example (2.1.9)]{GW}), where $a = \rmc(H) -1$ (see Setting \ref{setting}). Hence $\rmK_R \in \calM$ with $R \subseteq \rmK_R \subseteq V$ and therefore for each $n \in \Bbb Z$ we get
$$a-n \not\in H  \ \ \Longleftrightarrow  \ \ t^n \in \rmK_R.$$

For the rest of this section let us assume that $e=a_1 \ge 2$. We set $$
\alpha_i = \max \{n \in \Bbb Z \setminus H \mid n \equiv i \mod e\}
$$
for each $0 \le i \le e-1$ and  put
$
{\calS} = \{ \alpha_i \mid 1 \le i \le e-1 \}.
$
Hence $\alpha_0 = -e$, $\sharp {\calS}= e - 1$, $a = \max {\calS}$, and $\alpha_i \ge i$ for all $1 \le i \le e-1$. We then have the following.

\begin{fact}\label{4.4} $(1)$ $\rmK_R = \sum_{s \in {\calS}}Rt^{a-s}$. 

$(2)$ The set $\{t^{a-s} \mid s \in S \ \text{with}\ \fkm t^s \subseteq R\}$ forms a minimal system of generators of $\rmK_R$. 
\end{fact}

We now fix an ideal $I \in \calM$ such that $R \subseteq I \subseteq V$ and set $J = \rmK_R : I ~(\cong I^{\vee})$. Then $J \in \calM$ and $J \subseteq \rmK_R \subseteq V$. We assume that 
 the  canonical map
$$
t : I \otimes_RI^{\vee} \to \rmK_R, \ \ x\otimes f \mapsto f(x)
$$
is bijective. Then $1 \in J$ since $1 \in \rmK_R = IJ$, so that $R \subseteq J \subseteq \rmK_R$ and hence $R \subseteq I \subseteq \rmK_R$. We set $\mu_R(I) = r+1, ~\mu_R(J) = s+1~(r, s \ge 0)$ and let us write 
$
I = (t^{c_0},t^{c_1},\cdots, t^{c_r})$ and $J = (t^{d_0}, t^{d_1}, \cdots, t^{d_s})
$
with  integers $c_0 =0 < c_1 < \cdots < c_r$, $d_0 = 0< d_1 < \cdots < d_s$; hence
$$
\rmK_R = (t^{c_i +d_j} \mid 0 \le i \le r, ~0 \le j \le s)
$$
and $\{t^{c_i +d_j}\}_{0 \le i \le r, ~0 \le j \le s}$ is a minimal system of generators of $\rmK_R$, because $\mu_R(\rmK_R) = (r+1)(s+1)$. Therefore $t^{c_r + d_s} \not\in I \cup J$ when $r, s> 0$.

With this notation we have the following.

\begin{thm}\label{4.5} Let $b =\min \calS$ and suppose  $t^b \in R : \m$. Let $I \in \calM$ such that $R \subseteq I \subseteq V$. If the canonical map $t : I \otimes_RI^{\vee} \to \rmK_R$ is an isomorphism, then  $I \cong R$ or $I \cong \rmK_R$.
\end{thm}

\begin{proof}
Suppose that $r, s > 0$. Then $t^{c_r + d_s}{\cdot}J \not\subseteq \rmK_R$, since $t^{c_r + d_s} \not\in I=\rmK_R : J$. Choose $1 \le j \le s$ so that $t^{c_r + d_s + d_j} \not\in \rmK_R$. We then have $
a - (c_r + d_s + d_j) \in H.
$
Similarly 
$
a - (c_r + d_s + c_i) \in H
$
for some $1 \le i \le r$. 
Thanks to the uniqueness of  minimal systems of generators of the form $\{t^n\}_{n \in \Lambda}$ for a given  monomial ideal, by Fact \ref{4.4} the set $\{c_i + d_j \mid 0 \le i \le r,~0 \le j \le s\}$ is contained in $\{a-s \mid s \in \calS\}$, while  $t^{a-b}$ is a part of a minimal system of generators of $\rmK_R$, since $t^b \in R : \m$ but $t^b \not\in R$. Hence 
$
a - b = c_r + d_s,
$
as $b = \min \calS$.
Therefore $b - c_i, b-d_j \in H$. We set $\alpha = b -c_i,~ \beta = b -d_j$. Suppose that $\alpha \equiv \alpha_k \mod e$ for some $1 \le k \le e-1$. Then since $\alpha \in H$ but $\alpha_k \not\in H$, we get
$\alpha = \alpha_k + en$ for some $n \ge 1$, so that $a \le \alpha_k < \alpha = b - c_i$, because $b = \min {\calS}$. This is impossible. Hence $\alpha \equiv 0 \mod e$. We similarly have $\beta \equiv 0 \mod e$, whence 
$
c_i \equiv d_j \mod e,
$
which implies $t^{c_i} \in Rt^{d_j}$ or $t^{d_j} \in Rt^{c_i}$. This is also impossible, because $\{t^{c_i}, t^{d_j}\}$ is a part of a minimal system of generators of $\rmK_R$. Thus $r = 0$ or $s = 0$, whence $I \cong R$ or $I \cong \rmK_R$ as an $R$-module.
\end{proof}

The following is a special case of Theorem \ref{3.3a}. We note a proof in the present context.

\begin{cor}\label{4.6}
Suppose that $\rmv (R)=e$. Let $I \in \calM$ such that $R \subseteq I \subseteq V$. If the canonical map $t : I \otimes_RI^{\vee} \to \rmK_R$ is an isomorphism, then  $I \cong R$ or $I \cong \rmK_R$.
\end{cor}

\begin{proof}
It suffices to show $\m t^b \subseteq R$. Let $f = t^{e}$. Then $fR : \fkm = \fkm$, since $\m^2 = f\m$. Therefore  
$$
\mu_R(\rmK_R) = \ell_R((fR : \m)/fR)=\ell_R(\m/fR)=\ell_R(R/fR)-1=e-1,
$$ since $e = \ell_R(R/fR)$
(\cite[Lemma 3.1]{HK2}).
Consequently, as $\sharp{\calS} = e-1$, by Fact \ref{4.4}  $\{t^{a-s}\}_{s \in \calS}$ is a minimal system of generators of $\rmK_R$, so that $\m t^b \subseteq R$ as wanted.
\end{proof}

The condition $t^b \in R:\fkm$ does not imply $\rmv (R) = e$, as the following example shows.

\begin{ex}\label{4.7}
Let $H = \left<7, 22, 23, 25, 38, 40\right>$. Then ${\calS} = \{15, 16, 18, 33, 41\}$. We have $a = 41, b = 15$, and $\m{\cdot}t^{15} \subseteq R$, but $\rmv (R) = 6 < e= 7$.
\end{ex}


\section{The case where $\e (R) = 7$}

In this section we explore two-generated monomial ideals in numerical semigroup rings. We maintain Settings \ref{setting} and the notation in Section 4. Let $I \in \calM$ be a monomial ideal of $R$ such that $R \subseteq I \subseteq V$ and set $J = \rmK_R : I$. Suppose that $\mu_R(I) = \mu_R(J) = 2$ and write
$I =  (1, t^{c_1})$ and $J =  (1, t^{c_2})$, where $c_1,c_2>0$.
Throughout this section we assume:

\begin{condition}\label{5.0}
$IJ = \rmK_R$ and $\mu_R(\rmK_R) = 4.$
\end{condition}

\noindent
Hence $\rmK_R$ is minimally generated by $1, t^{c_1}, t^{c_2}, t^{c_1 + c_2}$.
Note that Condition \ref{5.0} is satisfied, once the canonical map $t : I\otimes_RI^{\vee} \to \rmK_R$ is an isomorphism.

We set  $c_3 = c_1 + c_2$. Then thanks to Fact \ref{4.4}, we may choose $b_1, b_2, b_3 \in {\calS}$ such that
$
c_1 = a - b_1, ~~c_2 = a - b_2, ~~c_3=a-b_3.
$
Hence $b_3=b_1 + b_2 -a$.

We begin with the following.

\begin{lem}\label{5.1}
The following assertions hold true.
\begin{enumerate}
\item[$(1)$] $a = b_1 +b_2 -b_3 \not\in H$.
\item[$(2)$] $2b_1 - a = b_1 + b_3 - b_2 \in H$.
\item[$(3)$] $2b_2 -a = b_2 + b_3 - b_1 \in H$.
\item[$(4)$] $b_2 + b_3 - a = 2b_3 - b_1 \in H$.
\item[$(5)$] $b_1 + b_3 - a = 2b_3 - b_2 \in H$.
\item[$(6)$] $2b_2 - b_3 \in H$.
\end{enumerate}
\end{lem}

\begin{proof} (1) This is clear.

(2)(3) Since $t^{c_1} \not\in J= \rmK_R : I$, we have $t^{c_1}I \not\subseteq \rmK_R$. Therefore by Fact  \ref{4.4} we see that $b_1 + b_3 -b_2 = 2b_1 - a = a - 2c_1 \in H$ because $t^{2c_1} \not\in \rmK_R$. Since $t^{c_2} \not\in I$, we similarly have $b_2 +b_3 - b_1 = 2b_2 -a = a - 2c_2 \in H$.

(4)(5) Since $t^{c_1 + c_2} \not\in I= \rmK_R : J$, we have $t^{c_1+ 2c_2} \not\in \rmK_R$. Therefore
$
2b_3 - b_1 = b_2 + b_3 -a = a - (c_1 + 2c_2) \in H.
$
Since $t^{c_1 + c_2} \notin J$, we have $t^{2c_1 + c_2} \not\in \rmK_R$. Hence
$
2b_3 - b_2 = b_1 + b_3 -a = a - (2c_1 + c_2) \in H.
$

(6) If $c_1 < c_2$, then $t^{c_2 - c_1} \not\in J= (1, t^{c_2})$, so that $t^{c_2 - c_1} \not\in \rmK_R$, because $t^{c_2-c_1}I \not\subseteq \rmK_R$. Hence $2b_2 - b_3 = a - (c_2 -c_1) \in H$. If $c_1 > c_2$, then $2b_2 - b_3 = a -(c_2 - c_1) > a = c - 1$, so that  $2b_2 - b_3 \in H$.
\end{proof}

Since $I = \rmK_R : (\rmK_R : I) = \rmK_R :J$ (\cite[Definition 2.4]{HK2}), we have a symmetry between $I$ and $J$. Hence without loss of generality we may  assume  $0 < c_1 < c_2$. Therefore 
$$
a > b_1 > b_2 > b_3 > 0\quad\text{and}\quad 2b_2 - a,~b_2 + b_3 -a, ~b_1 + b_3 - a \in H.
$$

\begin{lem}\label{5.2}
The following assertions hold true.
\begin{enumerate}
\item[$(1)$] $2b_2 \not\equiv b_1 + b_3 \mod e$.
\item[$(2)$] $2b_1 \not\equiv b_2 + b_3 \mod e$.
\end{enumerate}
\end{lem}

\begin{proof}
(1) Suppose $2b_2 \equiv b_1 + b_3 \mod e$. Then  $2b_1 -a = b_1 + b_3 -b_2 \equiv b_2 \mod e$. As $b_2 \not\in H$ but $b_1 + b_3 - b_2 \in H$, we have 
$
b_1 +b_3 - b_2 = b_2 + en
$
for some  $n \ge 1$. Hence $b_1 + b_3 - b_2 > b_2$, so that $b_1 > 2b_2 - b_3$. On the other hand, because $b_1 \not\in H$, $2b_2-b_3 \in H$ and $2b_2 -b_3 \equiv b_1 \mod e$, we get $2b_2 - b_3 > b_1$, which  is impossible.

(2) Assume $2b_1 \equiv b_2 + b_3 \mod e$. Then $b_2 + b_3 -b_1 \equiv b_1\mod e$. Since $b_2 + b_3 - b_1 \in H$ but $b_1 \not\in H$, we have $b_2 + b_3 -b_1 > b_1$, while $b_1>b_2>b_2+b_3-b_1$. This is absurd.
\end{proof}

\begin{prop}\label{5.3}
$e = a_1 \ge 8$.
\end{prop}

\begin{proof}
Since $4 = \mu_R(\rmK_R) \le \e(R) - 1$, we have $e =\e(R) \ge 5$. We consider the numbers
$$
a, ~b_1, ~b_2, ~b_3, ~2b_2 -a, ~b_2+b_3-a, ~b_1 +b_3 -a, ~2b_1 -a.
$$
Lemmata \ref{5.1} and \ref{5.2} show that  
$
2b_2 - a = b_2 + b_3 - b_1, ~b_2 + b_3 -a, ~b_1+b_3 -a \in H.
$
Therefore these numbers are distinct modulo $e$. Because these three numbers are less than $b_3$, the numbers 
$$
a, ~b_1, ~b_2, ~b_3, ~2b_2 -a, ~b_2+b_3-a, ~b_1 +b_3 -a
$$
are distinct modulo $e$. Thus $e \ge 7$.

Suppose that $e = 7$. Then $2b_2 - a \not\equiv 2b_1 - a \mod7$. Lemma \ref{5.2} (1) shows that $2b_1 - a = b_1 + b_3 - b_2 \not\equiv b_2\mod 7$. We have by Lemma \ref{5.2} (2) that 
$
b_2 + b_3 -a ~\not\equiv~2b_1 -a \mod 7,
$ which guarantees the following eight numbers
$$
a, ~b_1, ~b_2, ~b_3, ~2b_2 -a, ~b_2+b_3-a, ~b_1 +b_3 -a, ~2b_1 -a
$$
are distinct modulo $7$. This is absurd. Hence $e=a_1 \ge 8$.
\end{proof}

The goal of this section is Theorem \ref{1.8}. Let us restate it in our context.

\begin{thm}\label{5.4}
Let $R=k[[t^{a_1}, t^{a_2}, \cdots, t^{a_{\ell}}]]$ be a numerical semigroup ring over a field $k$ and suppose that $e = a_1 \le 7$. Let $I$ be a monomial ideal of $R$. If $I \otimes_RI^{\vee}$ is torsionfree, then $I \cong R$ or $I \cong \rmK_R$ as an $R$-module.
\end{thm}

\begin{proof}
Passing to the ring $B = I : I$, we may assume that the canonical map $t : I \otimes_R\Hom_R(I,\rmK_R) \to \rmK_R$ is an isomorphism. Suppose that $I \not\cong R$ and $I \not\cong \rmK_R$. Then 
$$
4 \le \mu_R(I){\cdot}\mu_R(I^\vee)= \mu_R(\rmK_R) = \rmr (R)\le e - 1 \le 6,
$$
(see \cite[Bemerkung 2.21 b)]{HK2}). If $\rmr (R) = 6$, then $\rmr (R) = e - 1$, so that $\fkm^2 = t^{e}\fkm$. In fact, since $t^eR$ is a reduction of $\fkm$, we get $e -1 = \ell_R(R/t^eR)-1 = \ell_R(\fkm/t^eR)$, while  $\rmr (R) = \ell_R((0):_{R/t^eR}\fkm)$. Therefore if $\rmr (R) = e - 1$, then $(0):_{R/t^eR}\fkm = \fkm/t^eR$, whence $\fkm^2 \subseteq t^eR$. Thus $\fkm^2 = t^e\fkm$, because $t^e \notin \fkm^2$. Consequently, if $\rmr (R) =6$, then $\rmv (R) =e = 7$, which violates Theorem \ref{3.3a} (and Corollary \ref{4.6} also), because $\rmr (R) = \mu_R(I){\cdot}\mu_R(I^\vee)$. Hence $\rmr (R) = 4$, so that $\mu_R(I)=\mu_R(I^\vee)=2$ which violates Proposition \ref{5.3}.
\end{proof}

\begin{cor}[{\cite[Main Theorem]{H1}}]\label{5.5}
Let $R$ be a Gorenstein numerical semigroup ring with $\e(R) \le 7$ and let $I$ be a monomial ideal in $R$. If $I \otimes_R\Hom_R(I,R)$ is torsionfree, then $I$ is a principal ideal.\end{cor}


\section{How to compute the torsion part $\rmT(I\otimes_RJ)$} 

Let $R$ be a Cohen--Macaulay local ring with $\dim R = 1$  and $F = \rmQ(R)$ the total ring of fractions. Let $I$ be a fractional ideal of $R$ and assume that $\mu_R(I) = 2$. In this section let us note an elementary  method to compute the torsion part $\rmT(I\otimes_RJ)$ of the tensor product $I\otimes_RJ$, where $J$ is a fractional ideal of $R$.  We need the method in Section 7 to explore concrete examples.

Let $f \in F$ and assume that $f \not\in R$. Let   $I= (1, f)~(= R + Rf)$ and choose a non-zerodivisor $\rho \in R$ so that $\rho I \subseteq R$. We set
$$
I' = \rho I,\quad\mathbf{a} =
\textstyle\binom{-f}{1}
\in F^2\quad\text{and}\quad R:I = \{ x \in F \mid xI \subseteq R \}.
$$
Recall that $R:I$ is also a fractional ideal of $R$ and $R:I \cong \Hom_R(I, R)$ as an $R$-module. 

Let 
$\varepsilon : R^2 \to I$ be the $R$-linear map defined by $\varepsilon
(\binom{a}{b})=a + bf$
for each $\binom{a}{b}\in R^2$. We then have the following.

\begin{fact}\label{6.1}
$\Ker~\varepsilon = \{ b \mathbf{a} \mid b \in R:I\} \cong R:I$.
\end{fact}

\noindent
Let $s = \mu_R(R:I)$ and write $R:I = (b_1,b_2, \ldots, b_s)$. We consider the exact sequence 
$$
R^{s} \stackrel{\Bbb M} \longrightarrow R^2 \xrightarrow{\tau = \rho [1, f]}R \longrightarrow R/{I'} \longrightarrow 0,
$$
where ${\Bbb M} = \left(
\begin{smallmatrix}
-b_1f & -b_2f & \cdots & -b_{s}f \\
b_1 & b_2 & \cdots & b_s
\end{smallmatrix}\right)$.
We now take an arbitrary fractional ideal $J$ of $R$. Then the homology $\rm{H}({\calC}) = {\rm Z}({\calC})/ {\rm B}({\calC})$ 
of the complex
$$
{\calC}: \ \ J^{\oplus s} \overset{\Bbb M}{\longrightarrow} J^{\oplus 2} \overset{\rho[1, f]}{\longrightarrow} J$$ gives $\Tor^R_1(R/{I'}, J)$ and since
$$
{\rm Z}({\calC}) = \{c \mathbf{a} \mid c \in J:I \} \cong J:I\ \ \text{and}\ \ 
{\rm B}({\calC}) = \{ c \mathbf{a} \mid c \in (R:I)J \} \cong (R:I)J,
$$
we have the isomorphism of $R$-modules

\begin{fact}\label{6.2}
$\Tor^R_1(R/{I'}, J) \cong (J:I)/{(R:I)J}$.
\end{fact}

We now consider the canonical isomorphisms
$$
\xi : J \to R\otimes_RJ, ~j \mapsto 1 \otimes j, \ \ \  
\eta : J^{\oplus 2} \to R^2 \otimes_R J,~\textstyle\binom{x}{y}\mapsto \textstyle\binom{1}{0}\otimes x + \binom{0}{1}\otimes y
$$
and consider the following diagram
\[
\xymatrix{
 & & 0& \\
 0 \ar[r] & \Tor^R_1(R/{I'}, J) \ar[r] & I' \otimes_R J \ar[r]^{\iota \otimes 1_J} \ar[u]& R \otimes_R J \ar[r] & R/{I'}\otimes_R J \ar[r] & 0 \\
 & &  & J \ar[u]_{\xi} & & \\
 & & R^2\otimes_R J \ar[uu]^{\rho \varepsilon \otimes_R 1_J} \ar[ruu]^{\tau \otimes_R 1_J} & J^{\oplus 2} \ar[l]_{\eta} \ar[u]_{\rho [1, f]}& \\
 & & R^{s}\otimes_R J \ar[u]^{{\Bbb M}\otimes_R 1_J}  & \\}
\] where the first row is exact and induced from  the short exact sequence $0 \to {I'} \overset{\iota}{\to} R \to R/{I} \to 0$ ($\iota$ denotes the embedding). We then have $\Tor^R_1(R/I', J) \cong \rmT(I' \otimes_R J)$ and
$\eta:
J^{\oplus 2} \to R^2\otimes_R J, \ \ c \mathbf{a} \longmapsto \textstyle\binom{1}{0}
\otimes (-c f) + \textstyle\binom{0}{1}
\otimes c
$
for each $c \in J:I$. Consequently, since $I \cong I' = \rho I$, we get the following isomorphism

\begin{prop}\label{6.3}
$(J:I)/(R:I)J \cong \rmT(I \otimes_R J),\ \ \overline{c} \longmapsto f\otimes c - 1 \otimes c f$ 
\end{prop}

\noindent
of $R$-modules, where  $\overline{c}$ denotes  the image of $c \in J:I$ in $(J:I)/(R:I)J$. In particular, setting $J = R:I$, we get the following.

\begin{cor}\label{b}
$\rmT(I\otimes_R(R:I)) \cong (R:I)^2/(R:I^2)$ as an $R$-module. 
\end{cor}

Let us examine a concrete example to test Corollary \ref{b}.

\begin{ex}\label{6.4}
We consider $H=\left<8,11,14,15\right>$ and $R=k[[t^8, t^{11}, t^{14}, t^{15}]]$. We take $I=(1,t)$. Then $R:I=(t^{14},t^{15}, t^{24}, t^{27})$ and $R:I^2=(t^{14}, t^{23}, t^{24}, t^{26}, t^{27})$. Since $(R:I)^2 = (t^{28}, t^{29}, t^{30}, t^{38})$, we have $t^{14} \notin (R:I)^2$, so that $I\otimes_R (R:I)$ has a non-trivial torsion $t\otimes t^{14} - 1 \otimes t^{15}$.
\end{ex}

\begin{proof}
The figure of $H$ is the following (the gray part).
\begin{center}
\begin{tabular}{cccccccc}
 &  & \\ 
\multicolumn{1}{>{\columncolor[gray]{0.8}}c}{0} & 1 & 2 & 3 & 4 & 5 & 6 & 7  \\ 
\multicolumn{1}{>{\columncolor[gray]{0.8}}c}{8} & 9 & 10 & \multicolumn{1}{>{\columncolor[gray]{0.8}}c}{11} & 12 & 13 & \multicolumn{1}{>{\columncolor[gray]{0.8}}c}{14} & \multicolumn{1}{>{\columncolor[gray]{0.8}}c}{15}  \\ 
\multicolumn{1}{>{\columncolor[gray]{0.8}}c}{16} & 17 & 18 & \multicolumn{1}{>{\columncolor[gray]{0.8}}c}{19} & 20 & 21 & \multicolumn{1}{>{\columncolor[gray]{0.8}}c}{22} & \multicolumn{1}{>{\columncolor[gray]{0.8}}c}{23}  \\ 
\multicolumn{1}{>{\columncolor[gray]{0.8}}c}{24} & \multicolumn{1}{>{\columncolor[gray]{0.8}}c}{25} & \multicolumn{1}{>{\columncolor[gray]{0.8}}c}{26} & \multicolumn{1}{>{\columncolor[gray]{0.8}}c}{27} & \multicolumn{1}{>{\columncolor[gray]{0.8}}c}{28} & \multicolumn{1}{>{\columncolor[gray]{0.8}}c}{29} & \multicolumn{1}{>{\columncolor[gray]{0.8}}c}{30} & \multicolumn{1}{>{\columncolor[gray]{0.8}}c}{31}  \\
\multicolumn{1}{>{\columncolor[gray]{0.8}}c}{32} & \multicolumn{7}{>{\columncolor[gray]{0.8}}c}{$\cdots$} \\
\end{tabular}
\end{center}
We have $c = 22$ and $a=21$. Since
$
R:I =(t^n \mid n \in H ~\text{such that}~n+1 \in H),
$
we get 
$R:I = (t^{14}, t^{15}) + {\c}$, where $\c = (t^n \mid n \ge 22)$. On the other hand, because $R:I^2 = (t^n \mid n \in H~\text{such that}~n+1, n+2 \in H)$, we have $R:I^2 = (t^{14}) + {\c}$. Hence $R:I = (t^{14}, t^{15}, t^{24}, t^{27})$ and $R:I^2 = (t^{14}, t^{23}, t^{24}, t^{26}, t^{27})$. Because $t^{14} \notin t^{28}V$ and $(R:I)^2 \subseteq t^{28}V$, we have $(R:I^2)/(R:I)^2 \neq (0)$ and Proposition  \ref{6.3} shows $I \otimes_R (R:I)$ has torsion $t\otimes t^{14} -1\otimes t^{15} \ne 0$.
\end{proof}


\section{Examples}

When $e=a_1 = 8$, there exists monomial ideals $I$ for which Condition \ref{5.0} is satisfied. However, as far as we know, for these ideals $I$ the $R$-modules $I\otimes_R I^{\vee}$ have non-zero torsions. Let us explore one example.

\begin{ex}\label{7.1}
We consider $H=\left<8,11,14,15\right>$ and  $R = k[[t^8,t^{11}, t^{14}, t^{15}]]$. Then $\rmK_R=(1, t, t^3,t^4)$. We take $I = (1,t)$ and set $J = \rmK_R : I$. Then $J= (1, t^3)$ and 
$$
\rmT(I\otimes_R J) = R(t\otimes t^{16}-1\otimes t^{17}) \cong R/\m.
$$
\end{ex}

\begin{proof}
Since ${\calS}=\{21, 20, 18,17,7,6,3\}$, we have $\rmK_R=\sum_{s \in S}Rt^{21-s} = (1, t, t^3, t^4)$. Let $I=(1, t)$. Then
$$
J = \rmK_R:I = (t^n \mid n \in {\Bbb Z}~\text{such~that}~21-n, 20-n \notin H),
$$
so that $J= (1, t^3)$. Hence $IJ= \rmK_R$, $\mu_R(I) = \mu_R(J) = 2$ and $\mu_R(\rmK_R)=4$, so that Condition \ref{5.0} is satisfied. 
Because
$$
R:I=(t^n\mid n\in {H}\ \text{and}\  n+1 \in H)\ \ \text{and}
$$
$$
J:I=(t^n\mid n\in {\Bbb Z}~\text{such that}~21-n,20-n, 19-n \notin H),
$$
we have
$$
R:I=(t^{14}, t^{15}, t^{24}, t^{27}),\ \ J:I=(t^{14}, t^{15}, t^{16}, t^{17}, t^{18}), \ \ (R:I)J=(t^{14}, t^{15}, t^{17}, t^{24}, t^{27}).
$$
Therefore $t^{16} \notin (R:I)J$ and $\m{\cdot}t^{16} \subseteq (R:I)J$. Thus Fact \ref{6.2} shows 
$$
\rmT(I\otimes_R J) \cong (J:I)/(R:I)J = R{\overline{t^{16}}} \cong R/\m, 
$$
where $\overline{t^{16}}$ is the image of $t^{16}$ in $(J:I)/(R:I)J$. Hence $0 \neq t \otimes t^{16}- 1\otimes t^{17} \in \rmT(I\otimes_R J)$ and 
$
\rmT(I\otimes_R J) \cong R/\m
$
as an $R$-module.
\end{proof}

\begin{rem}\label{7.2}
The ring $R$ of Example \ref{7.1} contains no monomial ideals $I$ such that $I \ncong R, I \ncong \rmK_R$, and $I\otimes_R I^{\vee}$ is torsionfree. 
\end{rem}

The following ideals also satisfy  Condition \ref{5.0} but $I\otimes_R I^{\vee}$ is not torsionfree. In fact, the semigroup rings of these numerical semigroups  contain no monomial ideals $I$ such that $I \ncong R, I \not\cong \rmK_R$, and $I\otimes_R I^{\vee}$ is torsionfree. 
\begin{enumerate}[$(1)$]
\item $H=\left<8,9,10,13\right>, \rmK_R=(1, t, t^3, t^4), I = (1, t)$.
\item $H=\left<8,11,12,13\right>, \rmK_R=(1, t, t^3, t^4), I = (1, t)$.
\item $H=\left<8,11,14,23\right>, \rmK_R=(1, t^3, t^9, t^{12}), I = (1, t^3)$.
\item $H=\left<8,13,17,18\right>, \rmK_R=(1, t, t^5, t^6), I = (1, t)$.
\item $H=\left<8,13,18,25\right>, \rmK_R=(1, t^5,t^7, t^{12}), I = (1, t^5)$.
\end{enumerate}

If $a_1 \ge 9$, then Theorem \ref{5.4} is no
longer true in general. Let us note one example. 

\begin{ex}\label{7.3}
Let $H=\left<9,10,11,12,15\right>$. Then $R=k[[t^9, t^{10}, t^{11}, t^{12}, t^{15}]]$. We have $\rmK_R =(1, t, t^3, t^4)$. Let $I=(1, t)$ and put $J=\rmK_R:I$. Then $J=(1, t^3)$, $\mu_R(I)= \mu_R(J)=2$, and $\mu_R(\rmK_R)=4$. We have $R:I=(t^9, t^{10},t^{11})$, $J:I=(t^9, t^{10}, t^{11}, t^{12}, t^{13}, t^{14})$, and $(R:I)J = J:I$, so that Proposition \ref{6.3} guarantees that $I\otimes_R I^\vee$ is torsionfree.
\end{ex}

\begin{ac}
The authors thank Olgur Celikbas for valuable information.
\end{ac}



\begin{thebibliography}{99}

\bibitem{A}
M. Auslander, {\it Modules over unramified regular local rings}, Illinois J. Math. {\bf 5} (1961), 631--647.

\bibitem{Constapel}
P. Constapel, {\it Vanishing of Tor and torsion in tensor products}, Comm. Algebra {\bf 24} (1996), 833--846.

\bibitem{Celikbas-Dao}
O. Celikbas and H. Dao, {\it Necessary conditions for the depth formula over Cohen--Macaulay local rings}, J. Pure Appl. Algebra (to appear).

\bibitem{CT}
O. Celikbas and R. Takahashi, {\it Auslander--Reiten conjecture and Auslander--Reiten duality}, J. Algebra {\bf 382} (2013), 100--114.

\bibitem{GL}
P. A. Garc{\'i}a-S{\'a}nchez and M. J. Leamer, {\it Huneke--Wiegand conjecture for complete intersection numerical semigroup rings}, J. Algebra {\bf 391} (2013), 114--124.

\bibitem{GW}
S. Goto and K.-i. Watanabe, {\it On graded rings I}, J. Math. Soc. Japan {\bf 309} (1978), 179--213. 

\bibitem{H1}
K. Herzinger, {\it Torsion in the tensor product of an ideal with its inverse}, Comm. Algebra {\bf 24} (1996), 3065--3083.

\bibitem{H2}
K. Herzinger, {\it The number of generators for an ideal and its dual in a numerical semigroup}, Comm. Algebra {\bf 27} (1999), 4673--4690.

\bibitem{HS}
K. Herzinger and R. Sanford, {\it Minimal generating sets for relative ideals in numerical semigroups of multiplicity eight}, Comm. Algebra {\bf 32} (2006), 4713--4731.

\bibitem{HK2}
J. Herzog and E. Kunz, Der kanonische Modul eines Cohen--Macaulay--Rings, Lecture Notes in Mathematics {\bf 238}, Springer--Verlag, 1971.

\bibitem{HW}
C. Huneke and R. Wiegand, {\it Tensor products of modules, rigidity and local cohomology}, Math. Scand. {\bf 81} (1997), 161--183.

\bibitem{L}
M. J. Leamer, {\it Torsion and tensor products over domains and specialization to semigroup rings}, \texttt{arXiv:1211.2896v1}.

\bibitem{R}
I. Reiten, {\it The converse to a theorem of Sharp on Gorenstein modules}, Proc. Amer. Math. Soc. {\bf 32} (1972), 417--420.

\bibitem{S1}
J. Sally, {\it On the associated graded ring of a local Cohen--Macaulay ring}, J. Math. Kyoto Univ. {\bf 19} (1977), 19--21.

\bibitem{S2}
J. Sally, Number of generators of ideals in local rings, Lecture Notes in Pure and Applied Mathematics {\bf 35}, Dekker, 1978.

\bibitem{S3}
J. Sally, {\it Cohen--Macaulay local rings of maximal embedding dimension}, J. Algebra {\bf 56} (1979), 168--183.

\end{thebibliography}
\end{document}